\documentclass[12pt]{amsart}
\usepackage{amsmath}
\usepackage{amsthm}
\usepackage{amsfonts}
\usepackage{amssymb}
\usepackage{tikz}
\newtheorem{Theorem}{Theorem}[section]
\newtheorem{Proposition}[Theorem]{Proposition}
\newtheorem{Lemma}[Theorem]{Lemma}
\newtheorem{Corollary}[Theorem]{Corollary}
\theoremstyle{definition}
\newtheorem{Definition}{Definition}[section]
\theoremstyle{remark}

\numberwithin{equation}{section}

\newcommand{\R}{{\mathbb R}}
\newcommand{\C}{{\mathbb C}}

\newcommand{\SL}{{\textrm{\rm SL}}}
\newcommand{\PSL}{{\textrm{\rm PSL}}}

\newcommand{\tr}{{\textrm{\rm tr}\:}}

\renewcommand{\Im}{{\textrm{\rm Im}\:}}
\renewcommand{\Re}{{\textrm{\rm Re}\:}}

\begin{document}

\title[Reflectionless operators]{Reflectionless operators and automorphic Herglotz functions}

\author{Christian Remling}

\address{Department of Mathematics\\
University of Oklahoma\\
Norman, OK 73019}
\email{christian.remling@ou.edu}
\urladdr{www.math.ou.edu/$\sim$cremling}

\date{December 12, 2024}



\begin{abstract}
I am interested in canonical systems and Dirac operators that are reflectionless on an open set. In this situation,
the half line $m$ functions are holomorphic continuations of each other and may be combined into a single function.
By passing to the universal cover of its domain, we then obtain a one-to-one correspondence of these operators with Herglotz functions
that are automorphic with respect to the Fuchsian group of covering transformations. I investigate the properties of this formalism, with
particular emphasis given to the measures that are automorphic in a corresponding sense.
This will shed light on the reflectionless operators as a topological space, on their extreme points, and on how the heavily studied smaller space of
finite gap operators sits inside the (much) larger space.
\end{abstract}
\maketitle
\section{Introduction}
This paper continues the theme of \cite{HMcBR,RemZ} along what I hope is a natural line of inquiry.
We consider canonical systems
\begin{equation}
\label{can}
J y'(x) = -zH(x)y(x) , \quad J = \begin{pmatrix} 0 & -1\\ 1&0 \end{pmatrix} ,
\end{equation}
with coefficient functions $H(x)\in\R^{2\times 2}$, $H(x)\ge 0$, $H\in L^1_{\textrm{loc}}(\R)$,
and Dirac equations
\begin{equation}
\label{Dirac}
Jy'(x) + W(x)y(x) = -zy(x) ,
\end{equation}
with $W(x)=W^t(x)\in\R^{2\times 2}$, $W\in L^1_{\textrm{loc}}(\R)$.
In both cases, these equations generate self-adjoint relations and operators on the associated Hilbert spaces
$L^2_H(\R)$ and $L^2(\R;\C^2)$, respectively.

The \textit{Titchmarsh-Weyl }$m$ \textit{functions }may be defined as
\begin{equation}
\label{defm}
m_{\pm}(z) = \pm y_{\pm}(0,z) ,
\end{equation}
and here $z\in\C^+=\{z\in\C : \Im z>0\}$ and $y_{\pm}(x,z)$ denotes the unique, up to a constant factor, solution $y$ of \eqref{can} or \eqref{Dirac} that
is square integrable on $\pm x>0$. On the right-hand side of \eqref{defm}, we also use the convenient convention of identifying
a vector $y=(y_1,y_2)^t\in\mathbb C^2$, $y\not= 0$, with the point $y_1/y_2\in\mathbb C_{\infty}$
on the Riemann sphere. So $m_{\pm}$ take values in $\C_{\infty}$, and in fact these functions are generalized \textit{Herglotz functions, }that is,
they map the upper half plane $\mathbb C^+$ holomorphically back to $\overline{\C^+}=\C^+\cup\R_{\infty}$.

Any generalized Herglotz function is the $m$ function of a canonical system on a half line, and if $H(x)$ is suitably normalized,
which is usually done by imposing the condition $\tr H(x)=1$, then we obtain a bijection between Herglotz functions and coefficient functions $H(x)$
\cite[Theorem 5.1]{Rembook}. If $H(x)\equiv P_{\alpha}$, the projection onto $e_{\alpha}=(\cos\alpha,\sin\alpha)$, on $x\ge 0$,
then $m_+(z)\equiv -\tan\alpha\in\R_{\infty}$,
and if $H$ is not of this special (trivial) type, then $m_+$ is a genuine Herglotz function, so maps $\C^+$ back to itself.

We can then also think of Dirac equations as special canonical systems since we can in particular realize the $m$ functions of a given Dirac operator
by a suitable canonical system. Of course, one can also rewrite \eqref{Dirac} directly. These issues are discussed in more detail in \cite[Section 2]{RemZ}.
When I talk about canonical systems in the sequel, it should be understood in this sense, that is, as containing Dirac equations as a special case.
In fact, we will mostly focus on Dirac operators in this paper, but occasionally the more general framework of canonical systems is useful.

We call a canonical system or a Dirac equation \textit{reflectionless }on a Borel set $A\subseteq\R$ if
\begin{equation}
\label{refless}
m_+(x) = -\overline{m_-(x)}
\end{equation}
for (Lebesgue) almost every $x\in A$. Reflectionless operators are important because they can be thought of as the basic building blocks of
arbitrary operators with some absolutely continuous spectrum \cite{RemAnn}, \cite[Chapter 7]{Rembook}.

Here we are interested in operators that are reflectionless on a \textit{finite gap set}
\[
U = \R_{\infty} \setminus \bigcup_{n=1}^N [a_n, b_n] , \quad a_1<b_1< \ldots < b_N .
\]
Right now, it may seem pointless and even strange to have put $\infty$ into this set but it will become clear in a moment why we did this,
when we introduce the function $M$ below.

So $U$ is an open subset of $\R_{\infty}$ whose complement has $N$
components, and all of these have positive length.
Much of what we do below applies more generally to any such set $U\subseteq\R_{\infty}$, whether or not $\infty\in U$, but I will not spell this out.

It will be convenient to have short-hand notations available for certain spaces of reflectionless operators. We set
\[
\mathcal R (U) = \{ H(x): H \textrm{ reflectionless on }U\} ,
\]
and we also write $\mathcal D$ for the collection of canonical systems that are Dirac operators, in the sense that their $m$ functions
$m_{\pm}$ are the $m$ functions of some Dirac equation \eqref{Dirac}. Furthermore, we let
\[
\mathcal D(U) = \mathcal D \cap \mathcal R(U) .
\]
Finally, we introduce
\[
\mathcal R_0(U) = \{ H\in\mathcal R(U): \sigma (H)\subseteq\overline{U} \}, \quad \mathcal D_0(U)=\mathcal D\cap\mathcal R_0(U) .
\]
Note that the trivial canonical systems $H\equiv P_{\alpha}$ are in $\mathcal R_0(U)$ (but not in $\mathcal D_0(U)$) according to this definition since they satisfy
$\sigma(H)=\emptyset$, $m_{\pm}(z)\equiv \mp\tan\alpha$.

The operators from $\mathcal D_0(U)$ are called finite gap operators, and they and especially their analogs for Schr{\"o}dinger operators
and Jacobi matrices have been studied extensively. See, for example, \cite{ConJo,GesHol,Mum,Teschl}

If a canonical system $H(x)$ is reflectionless on $U$, then
\[
M(z) = \begin{cases} m_+(z) & z\in\C^+ \\ -\overline{m_-(\overline{z})} & z\in\C^-
\end{cases}
\]
has a holomorphic continuation to $\Omega\equiv \C^+\cup U\cup\C^-$. Compare \cite[Lemmas 1.1, 1.2]{RemZ}.

Clearly, $M:\Omega\to\overline{\C^+}$ still takes values in the (closed) upper half plane.
If $N=1$, then $\Omega$ is conformally equivalent to $\C^+$, and a corresponding (explicit)
change of variable realizes $M$ as a Herglotz function, which we called the $F$ function of $H(x)$ in \cite{RemZ}.

If $N>1$, then $\Omega$ is no longer simply connected and thus not conformally equivalent to $\C^+$.
However, the universal cover can serve as a substitute, so we can make the following basic definition.
\begin{Definition}
\label{D1.1}
Let $\varphi: \C^+\to\Omega$ be the (unique) universal covering map with $\varphi(i)=\infty$, $\lim_{z\to i} -i(z-i)\varphi(z)>0$. Let $H\in\mathcal R(U)$.
Then we define the \textit{$F$ function }of $H$ as $F(\lambda)=M(\varphi(\lambda))$.
\end{Definition}
This choice of covering map
is motivated by \cite{CSZ1,SimRice}, which use an analogous map for different purposes. I take the covering space to be the upper half plane,
not the unit disk, as in \cite{CSZ1,SimRice}, which I find more convenient for my purposes, though that may be a matter of taste.
In any event, as we'll see, this map $\varphi$ will interact
nicely with the reflection symmetry of $\Omega$ about the real axis. Note also that we make $\varphi$ unique in the usual way by prescribing the value and the
argument of the derivative at a point, except that since this value is $\infty$, the second part about the derivative is slightly more awkward to write down than usual.

Our discussion so far has shown that $F$ is a Herglotz function, and it is clearly \textit{automorphic }with respect to the action of the group $G\le\operatorname{PSL}(2,\R)$
of covering transformations, that is of automorphisms $g$ of $\C^+$ satisfying $\varphi g=\varphi$: we have $Fg=F$ for all $g\in G$.
Here and in the sequel, we employ the convenient notational convention of writing composition of maps as juxtaposition, so, for example, $\varphi g=\varphi\circ g$.

We will denote the set of $G$ automorphic Herglotz functions by $\mathcal H_G$. It is a compact subset of $\mathcal H$, the space of all generalized
Herglotz functions. Here, we use the topology of locally uniform convergence on $\mathcal H$. In fact, this space is metrizable, and a possible choice of metric is
\begin{equation}
\label{defd}
d(F_1, F_2) = \max_{|z-i|\le 1/2} \delta (F_1(z), F_2(z)) ,
\end{equation}
with $\delta$ denoting the spherical metric, and we think of $\overline{\C^+}\subseteq\C_{\infty}$ as a subset of the Riemann sphere $S^2\cong\C_{\infty}$.

The space of all canonical systems also comes with a natural metric, which is discussed in detail in \cite[Section 5.2]{Rembook}.
We don't need an explicit description here. What matters for our purposes is the fact that if this metric is used, then the bijection
$H(x)\leftrightarrow (m_+(z),m_-(z))$ between canonical systems and pairs of Herglotz functions becomes a homeomorphism also \cite[Corollary 5.8]{Rembook}.

The following facts from \cite{RemZ} generalize to the automorphic setting without any difficulties.
\begin{Theorem}
\label{T1.1}
(a) The map $\mathcal R(U)\to\mathcal H_G$, $H\mapsto F(\lambda; H)$ that sends $H$ to its $F$ function
is a homeomorphism onto the set of automorphic Herglotz functions.

(b) Let $H\in\mathcal R(U)$. Then $H\in\mathcal D(U)$ if and only if $F(i;H)=i$.
\end{Theorem}
To make further progress, we'll have to study automorphic Herglotz functions and the associated measures in more detail, and this will
be our topic for much of this paper, in Sections 4--7. I hope that this analysis will be of some intrinsic interest. Returning to the spectral theory, we will then
be in a position to prove the following.
\begin{Theorem}
\label{T1.2}
$\mathcal D(U)$ is a compact convex set, and $H\in\mathcal D(U)$ is an extreme point if and only if $H\in\mathcal D_0(U)$.
\end{Theorem}
Here, the convexity refers to the natural linear structure on the $F$ functions, so for example the convex combinations of $H_1,H_2\in\mathcal D(U)$ are
the coefficients functions $H$ with $F$ functions $cF(z;H_1)+(1-c)F(z;H_2)$.

Theorem \ref{T1.2} implies that a continuous linear functional on $\mathcal D(U)$ assumes its extreme values on $\mathcal D_0(U)$.
An interesting example is provided by the Dirac potential itself, which indeed depends linearly on the $F$ function: we have the formula $F'(i)=c(q(0)+ip(0))$ \cite{RemZ},
and here we write $W=\bigl( \begin{smallmatrix}p & q\\q & -p\end{smallmatrix}\bigr)$, and we normalized $W$ by requiring that
$\tr W(x)=0$; compare \cite{LevSar,RemZ}. Since $\mathcal D(U)$, $\mathcal D_0(U)$ are invariant under shifts,
it follows that $\|W(x)\|$ for any $x\in\R$ is maximized by a $W\in\mathcal D_0(U)$ (we use sloppy but convenient notation here; it would be more formally
accurate to write $H\in\mathcal D_0(U)$, where $H=H_W$ is the associated canonical system). The following was proved in \cite{Remexp}, by different methods.
\begin{Theorem}[\cite{Remexp}]
\label{T1.3}
If $W\in\mathcal D(U)$, $\tr W(x)=0$, then
\[
\|W(x)\| \le \frac{1}{2} \sum_{n=1}^N (b_n-a_n)
\]
for all $x\in\R$. Equality at a single $x=x_0\in\R$ implies that $W\in\mathcal D_0(U)$.
\end{Theorem}
Another interesting linear functional on $\mathcal D(U)$ is given by the weights of certain associated measures. We'll discuss this briefly in Section 9.

We also obtain from Theorem \ref{T1.2} and Choquet's theorem a formula for the $F$ function (or $m_{\pm}$) of a general $H\in\mathcal D(U)$
as an average (or integral) of the corresponding functions of the $H\in\mathcal D_0(U)$, which have explicit representations; compare \eqref{hfcn} below.
We will not spell this out here.

Furthermore, the formalism developed in this paper sheds light on the topology of $\mathcal R(U)$ and especially $\mathcal D(U)$.
This is perhaps best discussed when we have a full understanding of the details, in Section 9, but let me at least mention here
that we will obtain, among other things, a natural homeomorphism
\begin{equation}
\label{1.7}
\mathcal D(U)\cong \mathcal M_1(S_1)\times\ldots\times\mathcal M_1(S_N) .
\end{equation}
Here, we can for now pretend that $S_n$ is a circle that is obtained by gluing together two copies of the $n$th gap $[a_n,b_n]$ at the endpoints
(later, $S_n$ will really be a preimage of this under $\varphi$);
$\mathcal M_1(X)$ denotes the space of (Borel) probability measures on $X$, with the weak $*$ topology. So $\mathcal M_1(S_n)$
is a compact metric space itself. I remark parenthetically that it is homeomorphic to the space $\{F\in\mathcal H: F(i)=i\}$, via the
Herglotz representation formula \eqref{hergl} below, and also to the Hilbert cube $[0,1]^{\mathbb N}$ \cite{Kel,Klee}. This last fact shows that
the spaces from \eqref{1.7} for different values of $N\ge 1$ are all homeomorphic to each other.

Nevertheless, it makes sense to set up a homeomorphism in this way. The map implicit in \eqref{1.7} interacts well with the linear structure on the
$F$ functions of the $H\in\mathcal D(U)$ that Theorem \ref{T1.2} refers to. Moreover, \eqref{1.7} provides a neat picture
of how to find $\mathcal D_0(U)$ inside the much larger space $\mathcal D(U)$.

It is well known and quite easy to show (we will review the argument in Section 5) that $\mathcal D_0(U)\cong T^N$, an $N$-dimensional torus.
As we will see in Section 7, the $H\in\mathcal D_0(U)$ correspond to the measures $(\delta_{x_1}, \ldots , \delta_{x_N})$, $x_n\in S_n$,
on the right-hand side of \eqref{1.7}. Since $\{\delta_x: x\in S_n\}\cong S_n$, this recovers the description of $\mathcal D_0 (U)$ as a torus
and shows how this space sits inside $\mathcal D(U)$.

\textit{Acknowledgments. }I thank Max Forester for help with preparing Figure 1 and Robert Furber for bringing to my attention the classical results \cite{Kel,Klee}
on the topology of $\mathcal M_1(S)$, in a Math Overflow answer.
\section{Fuchsian groups and the universal cover}
A \textit{Fuchsian group }$G$ may be defined as a discrete subgroup
\[
G\le\operatorname{PSL}(2,\R)=\SL(2,\R)/\{ \pm 1\} .
\]
The subject is, of course, classical. See, for example, \cite{Bea,Kat,Leh} for general introductions. Its use in spectral theory was
pioneered by Peherstorfer, Sodin, and Yuditskii \cite{PehYud,SodYud}. Further developments are due to Christiansen, Simon, and Zinchenko \cite{CSZ1,CSZ2}.
We will make heavy use of the basic formalism of \cite{CSZ1}; a reader friendly textbook style presentation is given in \cite{SimRice}.

It is perhaps worth mentioning that in the works cited above, the basic object is the meromorphic continuation of the $m$ function through the
complement of the essential spectrum, while we continue through the spectrum itself here. But this is perhaps one of the more superficial differences
since what we do here with the formalism
is quite different from how and for what purposes it is used in \cite{CSZ1,CSZ2,PehYud,SodYud}.

The domain $\Omega=\C^+\cup U\cup\C^-\subseteq\C_{\infty}$, viewed as a Riemann surface, is hyperbolic, that is,
it has the unit disk or, equivalently, the upper half plane
as its universal cover. We can obtain a unique covering map by prescribing the image and the argument of the derivative at a point, and this is what we did
in Definition \ref{D1.1}. A \textit{covering transformation }$g$ is an automorphism $g:\C^+\to\C^+$ satisfying $\varphi g=\varphi$. The automorphism
group of $\C^+$ is $\PSL(2,\R)$, with a matrix $g=\bigl ( \begin{smallmatrix} a & b \\ c & d \end{smallmatrix}\bigr)$ acting as a linear fractional transformation
\[
\begin{pmatrix} a & b \\ c & d \end{pmatrix} \cdot z = \frac{az+b}{cz+d} .
\]
This dot notation for group actions will be employed consistently in this paper. Also, we will never be very particular about the distinction between matrices $A$
and elements $g\in\PSL(2,\R)$, which are strictly speaking equivalence classes $g=\{ A,-A\}$. Of course, this is perfectly safe as long as we only perform
operations that are insensitive to an overall change of sign.

The group $G$ of covering transformations is a Fuchsian group. It is isomorphic to the fundamental group of $\Omega$, which is a free group
on $N-1$ generators. The non-identity elements $g\in G$ are \textit{hyperbolic, }that is, $|\tr g |>2$. Elliptic elements, $|\tr g|<2$, are ruled out
here because covering transformations do not have fixed points, and there are no parabolic elements, $|\tr g|=2$, because $\Omega$ does not have punctures.
Compare \cite[Section IV.9, Corollary 1]{FarKra}.

The behavior of $\varphi$ can be analyzed further by looking at the local inverse of $\varphi$ on $\C_{\infty}\setminus [a_1,b_N]$ with
$\varphi^{-1}(\infty)=i$. This is carried out in \cite{CSZ1,SimRice}. I refer the reader to these works for further details; as already mentioned,
\cite{CSZ1,SimRice} use the unit disk, not the upper half plane, as the covering space, so some small adjustments are necessary. I will now summarize those facts
that we will need below. In fact, much of this is best told by a picture.\\[1cm]
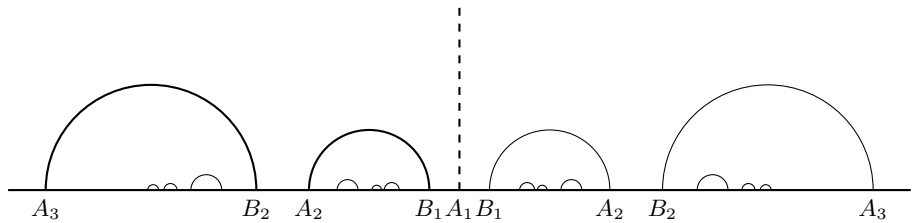
Figure 1: the covering map $\varphi$\\[1cm]
\begin{tikzpicture}[scale=1]

\draw[thick] (-6,0) -- (6,0) ;
\draw[dashed, thick] (0,0) -- (0,2.5);

\draw (.4,0) arc [start angle=180, end angle=0, radius=.8];
\draw (2.7,0) arc [start angle=180, end angle=0, radius=1.4];

\draw[thick] (-.4,0) arc [start angle=0, end angle=180, radius=.8];
\draw[thick] (-2.7,0) arc [start angle=0, end angle=180, radius=1.4];

\draw (3.16066,0) arc [start angle=180, end angle=0, radius=.20482];
\draw (3.75441,0) arc [start angle=180, end angle=0, radius=.08713];
\draw (4,0) arc [start angle=180, end angle=0, radius=.07202];

\draw (-3.16066,0) arc [start angle=0, end angle=180, radius=.20482];
\draw (-3.75441,0) arc [start angle=0, end angle=180, radius=.08714];
\draw (-4,0) arc [start angle=0, end angle=180, radius=.07202];

\draw (.8,0) arc [start angle=180, end angle=0, radius=.1];
\draw (1.03590,0) arc [start angle=180, end angle=0, radius=.06429];
\draw (1.34884,0) arc [start angle=180, end angle=0, radius=.13891];

\draw (-.8,0) arc [start angle=0, end angle=180, radius=.1];
\draw (-1.03590,0) arc [start angle=0, end angle=180, radius=.06429];
\draw (-1.34884,0) arc [start angle=0, end angle=180, radius=.13891];

\scriptsize
\draw[anchor=south] (0,-.5) node{$A_1$};
\draw[anchor=south] (.4,-.5) node{$B_1$};
\draw[anchor=south] (-.4,-.5) node{$B_1$};
\draw[anchor=south] (2,-.5) node{$A_2$};
\draw[anchor=south] (-2,-.5) node{$A_2$};
\draw[anchor=south] (2.7,-.5) node{$B_2$};
\draw[anchor=south] (-2.7,-.5) node{$B_2$};
\draw[anchor=south] (5.5,-.5) node{$A_3$};
\draw[anchor=south] (-5.5,-.5) node{$A_3$};
\end{tikzpicture}\\[1cm]
This describes the covering map $\varphi$ in the case $N=3$. The labels indicate images, so for example the point with label $A_1$, which is really just $z=0$,
has image $\varphi(A_1)=a_1$ etc. Here, we already make use of the fact that $\varphi$ can be extended from its original domain $\C^+$ through parts
of the real axis, which is discussed in more detail at the end of this section.
Similarly, the semicircles connecting $A_3$ and $B_2$, say, are mapped to $(b_2,a_3)$ under $\varphi$. The imaginary axis, depicted as a dashed line,
is mapped onto the subinterval $(b_3,a_1)\subseteq\R_{\infty}$. In particular, $\varphi(\infty)=b_3$.

The map $\varphi$ is symmetric about the imaginary axis in the sense that $\varphi(-\overline{z})=\overline{\varphi(z)}$. The exterior of the four large
circles, with the circles on the left included (shown in bold in the picture) but not the ones on the right, is a fundamental set:
it contains exactly one point from each orbit $G\cdot z$, $z\in\C^+$. I mention in passing that its interior is the Dirichlet region of $z_0=i$, that is, it contains
from each orbit that point that is closest to $z_0$ in the hyperbolic distance of $\C^+$.

The maps $g=IR$, with $R(z)=-\overline{z}$ being the reflection about the imaginary axis and $I$ denoting inversion about one of the large circles in
the right quarter plane, generate $G$. Here, an \textit{inversion }about the circle $|z-c|=r$ is defined as $I(z)=c+r^2/(\overline{z-c})$.

In the situation depicted in Figure 1, there are two such generators $g_1,g_2$. If we apply one of the transformations $g^{\pm 1}_j$, $j=1,2$,
to the four large circles, then we obtain one large circle and the three small next generation circles inside. The region bounded by these circles is another fundamental region.
The whole basic picture repeats itself on this smaller
scale, so the intervals on the real line in the closure of the smaller fundamental region are again mapped to the gaps $(a_j,b_j)$,
with each gap occurring twice in this way, corresponding to approach from above and below. Finally, all of this can of course be continued indefinitely,
by applying the generators of $G$ to these smaller circles etc.

The \textit{limit set }$L=L(G)$ can be defined as the collection of all limit points of the form $z=\lim g_n\cdot i$, with $g_n\in G$ being distinct
elements of $G$. We obtain the same set if we instead collect the limit points $\lim g_n\cdot z_0$ for any point $z_0\notin L$.
The limit set is closed and invariant under $G$. In our situation, $L\subseteq\R_{\infty}$ is a Cantor set if $N\ge 3$. If $N=2$,
then $G$ is cyclic and $L=\{ x,y\}\subseteq\R_{\infty}$ consists of the two fixed points of any $g\in G$, $g\not= 1$.

For our purposes here,
the following description will be extremely useful. As a preparation, we define $I_n\subseteq\R_{\infty}$ as essentially the preimage of
the $n$th gap $[a_n,b_n]$ in the closure of the large fundamental region.
More precisely, and referring to Figure 1, we define for example $I_1=[B_1,B_1)$, where by this nonsensical looking
expression we of course mean the half open interval between the two points with these labels. (For the purposes of describing the limit set,
the fine details are actually irrelevant and for example the closed interval would work too, but this attention to detail will pay off later.)
Similarly, $I_2=[B_2,A_2)\cup [A_2,B_2)$, and again the intended interpretation of this formula is clear from the figure.
Finally, $I_3=[A_3,A_3)\subseteq\R_{\infty}$, and here we start at the right point with label $A_3$ and move to the right through $\infty$ to the other point
labeled $A_3$. Later on, we will want to give each $I_j$ the topology of a circle in the natural way, but, as mentioned, this doesn't matter right now.
We then have
\[
L^c = \R_{\infty}\setminus L = \bigcup_{g\in G}g\cdot\mathcal F , \quad \mathcal F= \bigcup_{n=1}^N I_n ;
\]
compare \cite[eqn.\ (9.6.41)]{SimRice}. In fact, $\mathcal F$ is a fundamental set for $L^c$: if we had $g\cdot x =y$ for $x,y\in\mathcal F$,
then the mapping properties of the extended version of $\varphi$ that we will discuss below show at once that $x,y\in I_n$ would have to lie
in the same interval and in fact in its interior,
but this would then contradict the fact that the region outside the large circles is a fundamental region if we had $x\not= y$.

Finally, let's state precisely in what way exactly $\varphi$ can be extended past its original domain. The following is essentially a summary of
\cite[Theorem 9.6.4]{SimRice}, translated from the unit disk to the upper half plane.

We have a holomorphic extension
$\varphi: \C^+\cup L^c \cup\C^-\to\C_{\infty}$. This map is onto, and $\varphi(L^c)=\R_{\infty}\setminus U = \bigcup [a_j,b_j]$.
So the extension is by reflection $\varphi(z)=\overline{\varphi(\overline{z})}$, $z\in\C^-$. Obviously, the extended map can no longer be an unbranched covering
of its image $\C_{\infty}$. Rather, we have $\varphi'(z)=0$ precisely when $\varphi(z)= a_n$ or $=b_n$.
At these points, $\varphi''(z)\not= 0$. Finally, the extended map still satisfies $\varphi g=\varphi$ for all $g\in G$.
\section{Proof of Theorem \ref{T1.1}}
We pause the general development of the topic of the previous section to insert the rather routine proof of Theorem \ref{T1.1} here, but
will then return to it in the following section.

(a) The map $H\mapsto F(\cdot;H)$ is obviously injective since we can recover $m_{\pm}$ and thus also the canonical system itself (by \cite[Theorem 5.1]{Rembook})
from the $F$ function.

To prove that it is surjective, let $F_0\in\mathcal H_G$ be an arbitrary $G$ automorphic Herglotz function. To produce a canonical system
that has $F_0$ as its $F$ function we simply retrace the steps that lead from a given canonical system to its $F$ function. So define $M:\Omega\to\overline{\C^+}$,
$M(z)=F_0(\varphi^{-1}(z))$, and here we mean by $\varphi^{-1}$ any holomorphic local inverse of $\varphi$. These exist since $\varphi$ is a covering map. It doesn't
matter which local inverse we use here since $F_0$ is automorphic and the group $G$ of covering transformations acts transitively on the fibers $\varphi^{-1}(\{ z\})$.

Clearly, $M$ is holomorphic, and this function does map to $\overline{\C^+}$, as claimed. Next, let
\[
m_+(z) = M(z) , \quad m_-(z) = -\overline{M(\overline{z})} , \quad z\in\C^+ .
\]
These are Herglotz functions and thus there is a unique canonical system $H$ that has these functions as its half line $m$ functions.
The limits $m_{\pm}(x)=\lim_{y\to 0+} m_{\pm}(x+iy)$ exist for all $x\in U\subseteq\Omega$. Moreover, we can make both $\varphi^{-1}(x+iy)$
and $\varphi^{-1}(x-iy)$ approach the same point $\varphi^{-1}(x)\in\C^+$ (on one of the four large circles or the imaginary axis, say).
It follows that \eqref{refless} holds on $U$, and thus $H\in\mathcal R(U)$. By construction, $F(\lambda; H)=F_0(\lambda)$.

We have established that the map $\mathcal R(U)\to\mathcal H_G$, $H\mapsto F(\cdot; H)$, is a bijection.
Continuity in both directions is obvious since the topologies on both spaces
refer to locally uniform convergence and we are only changing variables. (It would actually be enough to confirm continuity of either the map
or its inverse since we are mapping between compact metric spaces, so have an automatic continuity result available.)

(b) This depends on an inverse spectral theory result and I don't want to get into the (unfortunately considerable)
intricacies of this topic here. Basically, the result holds because for $H\in\mathcal R(U)$, the $m$ functions are holomorphic at $z=\infty$,
and the condition $F(i)=i$ is equivalent to $m_{\pm}(\infty)=i$. This asymptotic behavior, in a generalized version, is a well known necessary condition for a Herglotz
function to be the $m$ function of a Dirac operator \cite{EHS}, and it is also sufficient here because we are dealing only with the specialized system from $\mathcal R(U)$.

In any event, the argument is identical
to the one presented in the proof of \cite[Theorem 3.2]{RemZ}; please see this reference for further details.
\section{Transformation of measures}
Herglotz functions $F$ have unique representations
\begin{equation}
\label{hergl}
F(z) = a + \int_{\R_{\infty}} \frac{1+tz}{t-z}\, d\nu(t) ,
\end{equation}
with $a=\Re F(i)\in\R_{\infty}$, and $\nu$ is a finite positive Borel measure on $\R_{\infty}$.

Let me state a few basic facts on how to recover $\nu$ from the boundary behavior of $\Im F(z)$, without attempting to give a systematic review.
We will make frequent use of these in the sequel. See for example \cite[Appendix B]{Teschl}, or pretty much any textbook on spectral theory, for a more comprehensive review
of the subject.

First of all, $F(t)\equiv\lim_{y\to 0+} F(t+iy)$ exists for (Lebesgue) almost all $t\in\R$. We have
\[
d\nu_{\textrm{\rm ac}}(t) = \frac{1}{\pi} \frac{\Im F(t)}{1+t^2}\, dt ,
\]
and the singular part $\chi_{\R}\nu_{\textrm s}$ is supported by
\[
\{t\in\R : \lim_{y\to 0+}\Im F(t+iy)=\infty \} .
\]
Finally, point masses correspond to pole type asymptotics; more precisely,
\[
(1+t^2)\nu (\{ t\} ) = \lim_{y\to 0+} -iyF(t+iy) , \quad \nu(\{\infty\}) =\lim_{y\to 0+} -iyF(i/y) .
\]

For any Herglotz function $F$ with
associated measure $\nu$, we define $\nu_g$ as the measure of $F g$. Since $\nu$ is already determined by the function $\Im F(z)$, the measure
$\nu_g$ indeed only depends on $\nu$ and $g$ and not on $a$ from \eqref{hergl}, as suggested by the notation.

We also denote by $g\nu$ the image measure $(g\nu)(B)=\nu(g^{-1}\cdot B)$. Recall that this obeys the substitution rule
\begin{equation}
\label{subst}
\int_{g\cdot A} f(y)\, d(g\nu)(y) = \int_A f(g\cdot x)\, d\nu(x) ,
\end{equation}
for $f\in L^1(\R_{\infty}, g\nu)$ or measurable $f\ge 0$.
\begin{Lemma}
\label{L3.1}
We have
\begin{equation}
\label{3.1}
d\nu_g(t) = \frac{\|w(g\cdot t)\|^2}{\|g^{-1}w(g\cdot t)\|^2}\, d(g^{-1}\nu)(t) , \quad w(x)\equiv \begin{pmatrix} x \\ 1 \end{pmatrix} .
\end{equation}
\end{Lemma}
This formula also works for $t=g^{-1}\cdot\infty$, if interpreted in the obvious way: we can let $w(\infty)=e_1$ in this case (notice
that multiplying $w$ by a factor will not affect the quotient). Of course, we can also
express the density on the right-hand side of \eqref{3.1} in terms of the entries of $g=\bigl( \begin{smallmatrix} a & b \\ c & d \end{smallmatrix} \bigr)$:
\[
\frac{\|w(g\cdot t)\|^2}{\|g^{-1}w(g\cdot t)\|^2} = \frac{(at+b)^2+(ct+d)^2}{t^2+1} .
\]
In any event, the slightly different function $f(g;x)$ from Lemma \ref{L3.2} below will be more important in the sequel.
\begin{proof}
Assume first that $F$ has a continuous extension to $\C^+\cup\R_{\infty}$. Then $Fg$ has the same property, and thus both $\nu$ and $\nu_g$
are purely absolutely continuous. As just reviewed, we have
\[
d\nu(t) = \frac{1}{\pi} \frac{\Im F(t)}{1+t^2}\, dt ,
\]
and similarly
\[
\nu_g(B) = \frac{1}{\pi} \int_B \frac{\Im F(g\cdot t)}{1+t^2}\, dt .
\]
Since $(d/dx)(g^{-1}\cdot x)=1/(a-cx)^2$, where we again write $g=\bigl( \begin{smallmatrix} a & b \\ c & d \end{smallmatrix} \bigr)$,
the substitution $x=g\cdot t$ gives
\begin{align*}
\nu_g(B) & = \frac{1}{\pi} \int_{g\cdot B}\frac{\Im F(x)}{(a-cx)^2\|w(g^{-1}\cdot x)\|^2}\, dx \\
& =\int_{g\cdot B}\frac{1+x^2}{(a-cx)^2\|w(g^{-1}\cdot x)\|^2}\, d\nu(x) \\
& = \int_{g\cdot B}\frac{\|w(x)\|^2}{\|g^{-1}w(x)\|^2}\, d\nu(x) .
\end{align*}
Now an application of \eqref{subst} with $A=g\cdot B$ and $g^{-1}$ taking the role of $g$ in \eqref{subst} lets us further rewrite this as
\[
\nu_g(B) = \int_B \frac{\|w(g\cdot t)\|^2}{\|g^{-1}w(g\cdot t)\|^2}\, d(g^{-1}\nu)(t) ,
\]
as desired.

The general case can then be handled by approximation. Given an arbitrary Herglotz function $F$, approximate it by functions $F_n$ of the type just discussed,
$d(F_n,F)\to 0$, with $d$ denoting the metric from \eqref{defd}; in other words, $F_n\to F$ locally uniformly. Then $\nu_n\to\nu$ in weak $*$ sense, that is,
$\int f\, d\nu_n\to\int f\, d\nu$ for all $f\in C(\R_{\infty})$, and here we give $\R_{\infty}$ its natural topology as the one-point compactification of $\R$ (so
$\R_{\infty}\cong S^1$, a circle). Clearly, we also have $d(F_ng,Fg)\to 0$, and thus $(\nu_n)_g\to\nu_g$ as well. By what we just established, we have
\begin{equation}
\label{3.2}
d(\nu_n)_g(t) = h(g;t)\, d(g^{-1}\nu_n)(t) , \quad h(g;t) = \frac{\|w(g\cdot t)\|^2}{\|g^{-1}w(g\cdot t)\|^2} .
\end{equation}
Since $g$ is a homeomorphism
of $\R_{\infty}$, the substitution rule \eqref{subst} makes it clear that the image measures converge to the expected limit $g^{-1}\nu_n\to g^{-1}\nu$.
Moreover, $h(g;t)$ is also continuous on $\R_{\infty}$, and thus \eqref{3.2} implies that
\[
d\nu_g(t) = \lim_{n\to\infty} d(\nu_n)_g(t) = h(g;t)\, d(g^{-1}\nu)(t) ,
\]
as claimed.
\end{proof}
\begin{Lemma}
\label{L3.2}
\begin{equation}
\label{3.3}
\nu_{g^{-1}}(g\cdot A)= \int_A f(g;x)\, d\nu(x) , \quad f(g;x)= \frac{\|w(x)\|^2}{\|gw(x)\|^2} .
\end{equation}
The density $f$ is a \textrm{\rm cocycle: }for any $g,h\in G$, $x\in\R_{\infty}$, we have
\begin{equation}
\label{cc}
f(gh;x) = f(g; h\cdot x)f(h;x) .
\end{equation}
\end{Lemma}
\begin{proof}
Of course, \eqref{3.3} just rephrases Lemma \ref{L3.1}; more precisely, this equation follows from \eqref{3.1} and the substitution
rule \eqref{subst}. The cocycle identity follows from a straightforward calculation, which I leave to the reader.
\end{proof}
\section{Automorphic measures}
\begin{Definition}
\label{D5.1}
We call a finite Borel measure $\nu$ on $\R_{\infty}$ \textit{automorphic }if $\nu=\nu_g$ for all $g\in G$. The collection of all automorphic measures
will be denoted by $\mathcal M_G$.
\end{Definition}
By Lemma \ref{L3.2}, one possible more explicit way of stating this condition is: $d(g^{-1}\nu)(x)=f(g;x)\, d\nu(x)$ for all $g\in G$.
One could of course also consider non-finite measures that are automorphic in this sense but they are useless for us here
since they do not occur in the Herglotz representation \eqref{hergl}.

The goal of this section and the next is to study automorphic measures in more detail and relate them to their restrictions to
the fundamental set $\mathcal F$.
Before we do this, it is perhaps worth clarifying explicitly the relation of Definition \ref{D5.1} to the $G$ invariance of the associated Herglotz functions.
\begin{Proposition}
\label{P5.1}
Let $F$ be a Herglotz function with representation \eqref{hergl}. Then the following statements are equivalent:\\
(a) $\nu=\nu_g$ for all $g\in G$;\\
(b) $(\Im F)g=\Im F$ for all $g\in G$;\\
(c) $Fg = F+a(g)$ for all $g\in G$, for certain constants $a(g)\in\R$;\\
(d) There is a homomorphism $\gamma: G\to (\R,+)$ such that $Fg=F+\gamma(g)$ for all $g\in G$.
\end{Proposition}
The parentheses in part (b) are unnecessary since $(\Im F)g =\Im (Fg)$; they emphasize that the condition can be read as stating
that the harmonic function $\Im F$ is automorphic.

The homomorphisms $\gamma: G\to\R$ form a vector space. Since $G$ is a free group on $N-1$ generators, $\{\gamma \}\cong\R^{N-1}$.
Our results below will show that any homomorphism $\gamma$ occurs as the $\gamma$ of a suitable $F$, as in part (d). See Section 7, especially Lemma \ref{L7.1}.
I mention these facts in passing; we will not use them here.
\begin{proof}
Since the measure $\nu$ is determined by and determines $\Im F$, which in turn determines $F$ up to a real constant, it is clear
that (a), (b), (c) are equivalent. If (c) holds, then of course $a(g)=Fg(z_0)-F(z_0)$ for any $z_0\in\C^+$, so
\begin{align*}
a(gh) &=F(gh)(i)-F(i) = Fg(h\cdot i)-F(h\cdot i) + Fh(i)-F(i) \\
& =a(g)+a(h) ,
\end{align*}
and $a=\gamma$ turns out to be a homomorphism, as claimed in part (d). The converse implication is trivial.
\end{proof}
\begin{Theorem}
\label{T5.1}
Suppose that $\nu\in\mathcal M_G$. Then $\nu(L)=0$.
\end{Theorem}
Recall that $L\subseteq\R_{\infty}$ denotes the limit set of $G$. This set is closed and invariant under $G$.
\begin{proof}
If $A\subseteq\R_{\infty}$ is a $G$ invariant set, then so is $A^c$, and now Lemma \ref{L3.1} shows that if $\nu$ is automorphic,
then so is $\chi_A\, d\nu$. We apply this remark to $A=L$ to obtain the automorphic measure $\chi_L\, d\nu$. Then we consider the corresponding harmonic function
\[
H(\lambda)=\Im \int_L \frac{1+t\lambda}{t-\lambda}\, d\nu(t) = \Im\lambda \int_L \frac{1+t^2}{|t-\lambda|^2}\, d\nu(t) , \quad\lambda\in\C^+ .
\]
Observe that $H$ has a continuous extension to $\C^+\cup (\R_{\infty}\setminus L)$, and $H=0$
on $\R_{\infty}\setminus L$.

Proposition \ref{P5.1}(b) shows that
$Hg=H$ for all $g\in G$. So as in the proof of Theorem \ref{T1.1}(a), we can define a harmonic function
$K:\Omega\to [0,\infty)$, $K(z)=H(\varphi^{-1}(z))$, using local inverses of $\varphi$.

Now the description of the mapping behavior of $\varphi$ from Section 2 shows that if $z\to x\notin\Omega$, then we can make
$\varphi^{-1}(z)$ approach $L^c$; more precisely, we can arrange that $\operatorname{dist}(\varphi^{-1}(z), \{ s,t\})\to 0$,
where $s,t$ are two preimages of $x$, for example in the same $I_n$. (We cannot guarantee convergence to a single limit because
preimages of two small semidisks above and below $x$ need not be close to each other.)

Since $H=0$ on $L^c$, no matter where exactly we are, this shows
that $K$ has a continuous extension to all of $\C_{\infty}$. This extended function has a maximum on its compact domain, and since
$K=0$ on $\C_{\infty}\setminus\Omega = \bigcup [a_j,b_j]$, the maximum is assumed at a point of $\Omega$, where $K$ is harmonic.
So $K, H\equiv 0$ and hence $\nu(L)=0$, as asserted.
\end{proof}

Theorem \ref{T5.1}, combined with Lemma \ref{L3.2}, shows that an automorphic measure is determined by its restriction to the fundamental set
$\mathcal F = \bigcup_{n=1}^N I_n$ for $L^c$. Lemma \ref{L3.2} also suggests the following procedure for constructing general automorphic measures: start out
with any finite measure on $\mathcal F$, and then propagate it to $L^c=G\cdot\mathcal F$ by using \eqref{3.3}, and now of course $\nu_{g^{-1}}=\nu$ if we want to
obtain an automorphic measure.

We must verify two things here: (1) The extended measure is still finite on $L^c$; (2) It is indeed automorphic.

The first point is addressed by Lemma \ref{L5.1} below, and the second one will be the content of Theorem \ref{T5.2}.
The question of whether series such as the ones from \eqref{5.1} converge, and for what exponents, is classical for general Fuchsian groups and has
been investigated extensively. This is no coincidence because our construction of automorphic measures is quite similar in spirit
to the construction of automorphic functions via Poincar\'e series. See for example \cite{Bea2,Pat}.

We will not rely on this theory here. In our more specialized situation, we can give a considerably simpler and perhaps more transparent treatment from scratch
if we study the issue not in isolation, as a question about general Fuchsian groups exclusively, but in the context of the spectral theory in which it arose here in the first place.
\begin{Lemma}
\label{L5.2}
For any $x\in L^c$, there is an $F\in\mathcal H_G$ whose associated automorphic measure satisfies $\nu( \{ x\} )>0$.
In fact, $F$ can be the $F$ function of an $H\in\mathcal D_0(U)$.
\end{Lemma}
This depends on a well known parametrization of $\mathcal D_0(U)$. This material is almost classical but since we will also need it in Section 7 below,
let me give a very quick review. The basic idea goes back to Craig \cite{Craig}. In the form used here, the method is discussed in detail
in \cite[Sections 2, 3]{ForRem}, but see also, for example, \cite{ConJo}.

Given an $H\in\mathcal D_0(U)$, let $h(z)=m_+(z)+m_-(z)$, and consider the Krein function of this Herglotz function, which is defined (almost everywhere) as
$\xi(t)=(1/\pi)\Im\log h(t)$, $0\le\xi(t)\le 1$. Since $H$ is reflectionless on $U$, we have $\Re h=0$ there, so $\xi=1/2$ on $U$. The further condition
$\sigma(H)\subseteq\overline{U}$ then implies that $\xi(t)=\chi_{(\mu_n,b_n)}(t)$ on each gap $a_n<t<b_n$, for some $\mu_n\in [a_n,b_n]$. The $\mu_n$ determine $\xi$
and thus already let us recover $h$ as
\begin{equation}
\label{hfcn}
h(z) = 2i \prod_{n=1}^N \frac{\sqrt{(a_n-z)(b_n-z)}}{\mu_n-z} ;
\end{equation}
compare \cite[eqn.\ (2.4)]{ForRem}. Note that $h$ has a pole at each $\mu_n$ which satisfies $\mu_n\not= a_n,b_n$, and this leads to a point mass
in the representing measure.

The function $h(z)$ does not uniquely determine $H$; rather, we must introduce the additional parameters $\sigma_n=\pm 1$.
These indicate whether this point
mass at $t=\mu_n$ is assigned to $m_+$ (if $\sigma_n=1$) or $m_-$ (if $\sigma_n=-1$). Other ways of splitting it are not allowed because that would
produce an eigenvalue at $\mu_n$, but $\sigma(H)\subseteq\overline{U}$ for $H\in\mathcal D_0(U)$.

If $\mu_n=a_n$ or $=b_n$, then
there is no such point mass and $\sigma_n$ becomes irrelevant. It is thus natural to combine $\mu_n,\sigma_n$
into one parameter $\widehat{\mu}_n=(\mu_n,\sigma_n)$ and
view $\widehat{\mu}_n$ as coming from a circle, which is obtained by gluing together two copies of the gap $[a_n,b_n]$ at the endpoints.

A more careful version of this analysis shows that the procedure sets up a homeomorphism between
$\mathcal D_0(U)$ and a torus $T^N$, by sending an $H\in\mathcal D_0(U)$
to its parameters $(\widehat{\mu}_1,\ldots , \widehat{\mu}_N)$.
\begin{proof}[Proof of Lemma \ref{L5.2}]
We can focus on the case $x\in\mathcal F$ since an automorphic measure with $\nu( \{ x\} )>0$ will also have point masses on the whole orbit $G\cdot x$.

For such an $x\in\mathcal F$, let $t=\varphi(x)$. Then $t\in\R_{\infty}\setminus U$, let's say $a_n\le t\le b_n$. If $t\not= a_n,b_n$, then introduce also
$\sigma =\operatorname{sgn} \Im \varphi(x+iy)$, $y>0$, $y$ small. So $\sigma=\pm 1$ informs us whether $\varphi(x+iy)$ approaches $t$ from above or below.

We now simply pick an $H\in\mathcal D_0(U)$ whose parameters $(\widehat{\mu}_1,\ldots , \widehat{\mu}_N)$ satisfy $\widehat{\mu}_n=(t,\sigma)$.
As discussed, the measure of $m_{\sigma}(z;H)$ will then have a point mass at
$t$ or, equivalently, $\lim_{y\to 0+} y\,\Im m_{\sigma}(t+iy)>0$. This means that also $\lim_{y\to 0+} y\,\Im M(t+i\sigma y)>0$ and this, in turn, implies
that $\lim_{y\to 0+} y\,\Im F(x+iy)>0$, as desired.

If $t=a_n,b_n$, the argument will not work in exactly this form because now the measures of $m_{\pm}$ have no point masses on $[a_n,b_n]$.
However, we can still take $\mu_n=t$, and this will make $m_{\pm}(z)\simeq 1/\sqrt{t-z}$ near $t$; compare \eqref{hfcn} above
and equation (2.8) from \cite{ForRem}.

Let's say $t=b_1$, to make this discussion more concrete, so that then $x$ is the (left) point with label $B_1$ in Figure 1.
In this situation, $m_{\pm}(z)$ have continuous extensions to $\C^+\cup (a_1,b_1)$ which are real on this interval (in fact, we can continue holomorphically,
but we won't need this here). Consider now a small semidisk
\[
D=\{ z: |z-x|<\delta, \Im z>0\} ,
\]
centered at $x=B_1$. The semicircle ending at $x$ from Figure 1 cuts this into two parts, and $F(\lambda)=m_+(\varphi(\lambda))$ on the left part while
$F(\lambda)=-\overline{m_-(\overline{\varphi(\lambda)})}$ on the right part, and of course $F$ is holomorphic on the full semidisk.
So what we just observed about the behavior of $m_{\pm}$ shows that $F$ similarly has a continuous extension to $\overline{D}\setminus\{ x\}$,
and $\Im F=0$ on $(x-\delta,x)$ and $(x,x+\delta)$. Hence $\nu$ gives zero weight to these sets.
If we had $\nu((x-\delta,x+\delta))=0$, then $F$ could in fact be extended continuously to this whole interval, but we already know that this isn't working
since $\lim_{y\to 0+} |F(x+iy)|=\infty$. Hence $\nu(\{ x\} )>0$, as desired.
\end{proof}
In this last part of the argument, it is really the same mechanism at work as in the easier case $t\not= a_n,b_n$, only in a more elaborate version:
$M$ has a square root type singularity, which does not lead to a point mass in the associated measure, but then $\varphi(x+h)-t\simeq h^2$,
so the change of variable that happens when we move from $M$ to $F$ amplifies this and we do end up with the required pole type behavior that
corresponds to a point mass.

In the next result,
we again write a matrix representing a general $g\in G$ as $g=\bigl( \begin{smallmatrix} a & b\\ c&d \end{smallmatrix}\bigr)$.
\begin{Lemma}
\label{L5.1}
We have
\begin{equation}
\label{5.1}
\sum_{g\not= 1} \left( \frac{1}{a^2} + \frac{1}{b^2} + \frac{1}{c^2} + \frac{1}{d^2} \right) < \infty .
\end{equation}
Moreover, $\sum_{g\in G} f(g;x)$ converges locally uniformly on $x\in L^c$.
\end{Lemma}
The result in this form depends on our specific choices for the covering map (and thus the Fuchsian group), which make sure that $0,\infty\notin L$.
It fails badly for example for the cyclic group with generator $g\cdot z =2z$ since then $b=c=0$ for all $g\in G$ and, to add insult to injury,
$\{ a^2\} =\{ d^2\} = \{ 2^n: n\in\mathbb Z\}$.
\begin{proof}
We start by observing that $a\not= 0$ for all $g\in G$, because otherwise $g\cdot\infty=0$, which is impossible because the extended version of $\varphi$
satisfies $\varphi(\infty)\not= \varphi(0)$. For the same reason, $d\not= 0$ for all $g\in G$. Similarly, if we had $b=0$ or $c=0$
for a $g\not= 1$, then $z=0$ or $z=\infty$ would be a fixed point of $g$, but fixed points are in $L$; in fact, $L$ can also be obtained as the closure of
the fixed points \cite[Theorem 3.4.4]{Kat}.

We have uniform bounds on the quotients
\begin{equation}
\label{5.2}
0<C_1\le \left| \frac{a}{b} \right|, \left| \frac{a}{c} \right|,\left| \frac{b}{d} \right|,\left| \frac{c}{d} \right| \le C_2, \quad g\in G, g\not= 1 .
\end{equation}
Indeed, if we had, let's say, $a_n/b_n\to 0$ for certain $g_n\in G$, then $g^{-1}_n\cdot 0\to\infty$, which would imply that $\infty\in L$.
Similar arguments establish the other bounds.

We know from Lemma \ref{L5.2} that there is an automorphic measure with $\nu(\{ 0\})=1$. Now \eqref{3.3} shows that $\nu(\{ g\cdot 0\})=f(g;0)$,
and these points $g\cdot 0$, $g\in G$, are distinct because fixed points are in $L$ and $0\notin L$. Thus
$\sum f(g;0)<\infty$. We have $f(g;0)=1/(b^2+d^2)$, and then \eqref{5.2} implies that $\sum 1/X^2<\infty$ for $X=a,b,c,d$.

Next, I claim that $\sum_{g\in G} f(g;x)$ converges uniformly on
\[
|x|\ge\max \{ 2/C_1,1\} ,
\]
and here $C_1>0$ is the constant from \eqref{5.2}.
Indeed, for these $x$, we have $|b/a|\le |x|/2$, so $(ax+b)^2\ge a^2x^2/4$ and hence
\[
f(g;x) = \frac{x^2+1}{(ax+b)^2+(cx+d)^2} \le \frac{x^2+1}{(ax+b)^2} \le \frac{8}{a^2} .
\]
The uniform convergence of $\sum f(g;x)$ now follows from the convergence of $\sum 1/a^2$.

So, to establish uniform convergence on an arbitrary compact subset $K\subseteq L^c$, it now suffices to consider the case
$\infty\notin K$, so $K$ is a bounded subset of $\R$. Then obviously
\[
f(g;x) \le \frac{C}{(ax+b)^2}
\]
for $g\in G$, $x\in K$.

There can be only finitely many $g\in G$ with $-b/a\in K$ because $-b/a=g^{-1}\cdot 0$
and thus the existence of an infinite sequence of such $g$'s would give us a limit point in $K$.
We can then estimate $(ax+b)^2\gtrsim a^2$, uniformly in $x\in K$ and $g\in G\setminus F$,
with $F$ denoting this finite (possibly empty) set of $g\in G$ with $-b/a\in K$. Indeed, suppose that on the contrary we had
\[
\frac{(a_nx_n+b_n)^2}{a_n^2} = \left( x_n + \frac{b_n}{a_n}\right)^2\to 0
\]
for certain $g_n\in G\setminus F$, $x_n\in K$.
We can also assume here that $x_n\to x$, and it then follows that $x=-\lim b_n/a_n=\lim g^{-1}_n\cdot 0$. Since $x\notin L$, this would imply
that the $g_n$ come from a finite set, but then $-b_n/a_n=x\in K$ for large $n$ even though we specifically avoided the $g$ satisfying this condition.
So we again have a uniform estimate $f(g;x)\lesssim 1/a^2$ for $g\in G$, $x\in K$.
\end{proof}
\begin{Corollary}
\label{C5.1}
Let
\[
D(x) = \sum_{g\in G} f(g;x) .
\]
Then $D(x)$ is a continuous positive function on the open set $x\in L^c$, and $D(x)=\infty$ for all $x\in L$.
\end{Corollary}
Since it is also clear from this that $D(x)$ gets large when $x\in L^c$ approaches a point of $L$, we can rephrase the statement as follows:
$D:\R_{\infty}\to [1,\infty]$ is continuous and $D^{-1}(\{ \infty\}) = L$.
\begin{proof}
Lemma \ref{L5.1} immediately implies that $D$ is continuous on $L^c$ and $1\le D<\infty$ there.

We will not need the second part here, about the behavior of $D$ on $L$, so I will just sketch the argument. First of all, if $x_0\in L$ is a fixed point of a
$g_0\in G$, $g_0\not= 1$, then $g_0w(x_0)=\lambda w(x_0)$ and thus $f(g^n_0;x_0)=|\lambda|^{-2n}$.
Hence $\sum_{n\in\mathbb Z} f(g^n_0;x_0)$ already diverges.

Suppose now that $x_0\in L$ is not a fixed point of any $g\in G$, $g\not= 1$, so that the points $g\cdot x_0$, $g\in G$, are all distinct.
If we had $D(x_0)<\infty$, then the measure
\[
\nu = \sum_{g\in G} f(g;x_0)\delta_{g\cdot x_0}
\]
would be finite. It is also automorphic, and this we can see in the same way as below, when we discuss this construction in general. See the proof of Theorem \ref{T5.2} below.
We have reached a contradiction because we already know from Theorem \ref{T5.1} that automorphic measures can not give weight to $L$.
\end{proof}
We now have all the tools needed to carry out the construction of automorphic measures that was already described above.
Given a finite Borel measure $\nu_0$ on $\mathcal F=\bigcup I_n$ and a Borel set $A\subseteq\mathcal F$, let
\begin{equation}
\label{5.3}
\nu(g\cdot A) = \int_A f(g;x)\, d\nu_0(x) .
\end{equation}
For fixed $g\in G$, this defines a measure on $g\cdot\mathcal F$. Now since $\mathcal F$ is a fundamental set for $L^c$,
the rule \eqref{5.3} also uniquely determines a Borel measure $\nu$ on all of $L^c$ (or on $\R_{\infty}$, with $\nu(L)=0$). More explicitly,
if $B\subseteq L^c$ is an arbitrary Borel set, then we write $B=\bigcup A_g$ as a countable disjoint union of the sets $A_g=B\cap g\cdot\mathcal F$,
and we set $\nu(B)=\sum \nu(A_g)$, with the summands defined via \eqref{5.3}. Corollary \ref{C5.1} shows that
\begin{equation}
\label{5.11}
\nu(L^c) = \int_{\mathcal F} D(x)\, d\nu_0(x) < \infty
\end{equation}
since $\mathcal F$ is contained in a compact subset of $L^c$, so $D$ is bounded there.
\begin{Theorem}
\label{T5.2}
For any finite Borel measure $\nu_0$ on $\mathcal F$, the measure $\nu$ defined above is automorphic.
\end{Theorem}
\begin{proof}
We will verify that $\nu_{g^{-1}}=\nu$ for all $g\in G$. As above, given a Borel set $B\subseteq L^c$, we can write $B=\bigcup h\cdot C_h$, $C_h\subseteq\mathcal F$,
so it suffices to check that $\nu_{g^{-1}}(h\cdot A)=\nu(h\cdot A)$ for $A\subseteq\mathcal F$.

By Lemma \ref{L3.2} and the substitution rule \eqref{subst}, we have
\[
\nu_{g^{-1}}(h\cdot A) = \int_{g^{-1}h\cdot A} f(g;x)\, d\nu(x) =\int_A f(g;g^{-1}h\cdot t)\, d(h^{-1}g\nu)(t) .
\]
Also, $d(k^{-1}\nu)(t) = f(k;t)\, d\nu_0(t)$ on $\mathcal F$ by the defining property \eqref{5.3} of $\nu$. Thus the cocycle identity \eqref{cc} shows that
\begin{align*}
\nu_{g^{-1}}(h\cdot A) &= \int_A f(g;g^{-1}h\cdot t)f(g^{-1}h;t)\, d\nu_0(t) \\
& = \int_A f(h;t)\, d\nu_0(t) = \nu(h\cdot A)
\end{align*}
as desired.
\end{proof}
\section{Automorphic measures and their restrictions to $\mathcal F$}
Of course, \eqref{5.3} and the above procedure were forced on us by Lemma \ref{L3.2}, so $\nu$ is the only automorphic measure
whose restriction to $\mathcal F$ is $\nu_0$. We have set up a bijection between arbitrary finite measures on $\mathcal F$ and automorphic measures
on $L^c$ or $\mathbb R_{\infty}$. We want to go further here and show that this map is a homeomorphism if we use the weak $*$ topology on measures.
Here we give the subintervals
$I_n\subseteq\mathcal F$ the topology of a circle $S^1$. (If this seems unmotivated at this point, please see the discussion following the statement of
Theorem \ref{T6.1} below for an explanation of why this is important.)

This latter topology matters even if we are only interested in the measures because if $J=[C,D)$ is a half-open interval,
then for example $\delta_{D-1/n}\to\delta_C$ if $J$ is given the circle topology.

To constantly remind ourselves that this is the topology on $\mathcal F$ we are using, we now denote the subintervals of this
set by $S_n$. So $S_n=I_n$ as a set, and for example $S_1$ can be viewed as the segment between the two points with label $B_1$ in Figure 1 of Section 2,
with these points identified. To obtain $S_2$, we must first glue together the two pieces $[B_2,A_2)$ and $[A_2,B_2)$ of $I_2$ and then again
endow the resulting interval $[B_2,B_2)$ with the circle topology. Finally, in the situation of Figure 1, the last set $I_3$ is again already a single interval since
it goes through the point $\infty\in\R_{\infty}$, which we can't see in the picture. We glue together the two points labeled $A_3$ to produce $S_3$.

Recall that we denote the space of automorphic measures by $\mathcal M_G$. The almost self-explanatory notation $\mathcal M (\mathcal F)$
will similarly refer to the space of finite measures on $\mathcal F$. In both cases, we use the weak $*$ topology, and, as discussed, we give
$\mathcal F\cong S_1\sqcup \ldots \sqcup S_N$ the topology of a disjoint union of $N$ circles. We then also have an obvious (linear) homeomorphism
\[
\mathcal M(\mathcal F) \cong \mathcal M(S_1) \times \ldots\times\mathcal M(S_N) ,
\]
and we switch between these two realizations of the space without further comment when convenient.
\begin{Theorem}
\label{T6.1}
The restriction map $\mathcal M_G\to\mathcal M(\mathcal F)$, $\nu\mapsto (\nu_1,\ldots ,\nu_N)$, $\nu_n(A)=\nu(A)$, $A\subseteq S_n$, is a
homeomorphism.
\end{Theorem}
We can now appreciate the significance of the circle topology in this context by running a quick informal check on this result. Observe first of all that the restriction map
is not continuous on general, not necessarily automorphic measures.
For example, we can again consider $\nu=\delta_x$, with $x$ converging from the left to (let's say) the left endpoint $B_1$ of $I_1$. Then all restrictions to $I_1=S_1$
are zero, but the sequence converges to $\delta_{B_1}$, which is equal to its restriction.

This kind of example would also be a problem for automorphic measures if we weren't using the circle topology. With the circle topology, all will be well
because such a point mass will lead to a corresponding point mass near the right endpoint of $I_1$, and this measure is now close to $\delta_{B_1}$.

The following observation will make sure that the weights come out right in this type of example.
\begin{Lemma}
\label{L6.2}
Consider one of large semicircles $|z+c|=r$ from Figure 1 (or, more formally, a suitable preimage under $\varphi$ of a subinterval $(b_j,a_{j+1})\subseteq U$).
Let $g\in G$ be the unique map that maps this semicircle to its reflected version $|z-c|=r$.
Then, for $x=-c\pm r$, we have $g\cdot x =-x$ and $f(g;x)=1$.
\end{Lemma}
\begin{proof}
Recall from Section 2 that we have a description of $g=IR$ as the composition of reflection about the imaginary axis with an inversion
about $|z-c|=r$. More explicitly,
\[
g\cdot z = c -\frac{r^2}{z+c} = \frac{cz +c^2-r^2}{z+c} .
\]
So as the $\SL (2,\R)$ matrix representing $g$ we can take
\[
g = \frac{1}{r} \begin{pmatrix} c & c^2-r^2 \\ 1 & c \end{pmatrix} ,
\]
and given this information, a straightforward calculation will now finish the proof.
\end{proof}
\begin{proof}[Proof of Theorem \ref{T6.1}]
As discussed earlier, we already know that the map is a bijection. We must still establish that both the map and its inverse are continuous.

We start with the restriction map itself, and here
we'll discuss explicitly only the map $\nu\mapsto\nu_1$. The general case $\nu\mapsto\nu_n$ is of course similar but more awkward to write down when
$I_n$ has two pieces that need to be glued together.

Let's write $I_1=I=[c,d)$. As a preliminary, recall that the weak $*$ topology is metrizable on the bounded open subsets $\nu(\R_{\infty})<A$,
and it suffices to prove continuity on these sets.
We must thus show that if $\nu_n\to\nu$ in $\mathcal M_G$ and $f\in C[c,d]$, $f(c)=f(d)$, then
\begin{equation}
\label{6.1}
\int_{[c,d)} f(x)\, d\nu_n(x) \to \int_{[c,d)} f(x)\, d\nu(x) ,
\end{equation}
and here the hypothesis $\nu_n\to\nu$ means that similarly
\begin{equation}
\label{6.2}
\int_{\R_{\infty}}g(x)\, d\nu_n(x)\to \int_{\R_{\infty}} g(x)\, d\nu(x)
\end{equation}
for all $g\in C(\R_{\infty})$.

To do this, write $f(x)=f(c)+h(x)$. Then $h\in C[c,d]$ can be continuously extended to $\R_{\infty}$ by setting $h(x)=0$ for $x\notin I$, so
\eqref{6.1} for $f=h$ is essentially the same as our assumption \eqref{6.2}.

So we need only show that $\nu_n([c,d))\to\nu ([c,d))$. This is clear if $\nu(\{c\})=\nu(\{d\})=0$.

In general, we can move the endpoints $c,d$ slightly to avoid possible point masses. If we do this judiciously, we can exploit the fact that the measures
are automorphic to make sure that the measure doesn't change much when we adjust the interval in this way.

Here are the details.
Let $\epsilon>0$ be given. Let $g_0$ be the map from Lemma \ref{L6.2} for the circles ending
at $c$ and $d$, respectively. So $g_0\cdot c=d$ and $f(g_0;c)=1$. Now pick a small $\delta_1>0$ and define $\delta_2>0$ by $g_0\cdot (c-\delta_1)= d-\delta_2$.
Then $\delta_2$ will be small as well and in fact $\delta_2$ will also decrease strictly as $\delta_1>0$ approaches zero.
Thus, by taking $\delta_1$ small enough and avoiding the at most countably many point masses, we can make sure that the following statements will hold: 
\begin{gather}
\label{6.5}
1-\epsilon \le f(g_0;x)\le 1+\epsilon \;\textrm{ for }c-\delta_1\le x\le c , \\
\label{6.3}
\nu(\{ c-\delta_1\}) =\nu(\{ d-\delta_2\}) = 0 .
\end{gather}
Notice that $g_0\cdot [c-\delta_1,c) =[d-\delta_2,d)$, so \eqref{6.5}, combined with
Lemma \ref{L3.2}, shows that
\[
(1-\epsilon )\mu([c-\delta_1,c)) \le \mu([d-\delta_2,d))\le (1+\epsilon) \mu([c-\delta_1,c)) \textrm{ for all }\mu\in\mathcal M_G .
\]
In particular, we have
\begin{equation}
\label{6.6}
\left| \nu_n([c-\delta_1,c))-\nu_n([d-\delta_2,d))\right| \le \epsilon\nu_n(\R_{\infty})\le C\epsilon ,
\end{equation}
and of course the same estimate holds for $\nu$. Moreover, \eqref{6.3} implies that $\nu_n([c-\delta_1,d-\delta_2))\to \nu([c-\delta_1,d-\delta_2))$,
and if we combine this with \eqref{6.6}, we see that
\[
\left| \nu_n([c,d))-\nu([c,d))\right| < (2C+1)\epsilon
\]
for all sufficiently large $n$.

To prove that the inverse map is continuous, we again restrict the (original) map to the subsets $\{\nu\in\mathcal M_G: \nu(\R_{\infty})\le C\}$.
These are closed subsets of the compact metric spaces $\{\nu\in\mathcal M(\R_{\infty}): \nu(\R_{\infty})\le C\}$ and thus compact themselves.
To confirm this, assume that the $\nu_n$
are automorphic measures and $\nu_n\to\nu$. Then the associated Herglotz functions, normalized by requiring that $\Re F_n(i)=\Re F(i)=0$
converge similarly $F_n\to F$, and the functions $\Im F_n$ are automorphic by Proposition \ref{P5.1}(b). Clearly a pointwise
limit of automorphic functions is automorphic itself, so by referring to Proposition \ref{P5.1} for a second time, we see that
$\nu\in\mathcal M_G$, as claimed.

So we now have a continuous map between the compact metric spaces $\{\nu\in\mathcal M_G: \nu(\R_{\infty})\le C\}\to\{\nu\in\mathcal M(\mathcal F):
\nu(\mathcal F)\le C\}$, and it is a bijection onto its image. The continuity of the inverse is automatic in this situation. We can now finish the proof
by recalling \eqref{5.11} and the fact that the function $D(x)$ from that identity is bounded on $\mathcal F$. This makes sure that for any $A>0$,
the set $\{\nu\in\mathcal M(\mathcal F): \nu(\mathcal F)<A\}$ is in the image of the restricted maps discussed above if we choose $C>0$ large enough.
\end{proof}
\section{From automorphic measures to automorphic Herglotz functions}
Theorems \ref{T5.1}, \ref{T5.2}, and \ref{T6.1} give us a description of the automorphic measures, and we must now clarify which of these will lead
to automorphic functions when used in \eqref{hergl}.
Notice that Proposition \ref{P5.1}(d) associates a homomorphism $\gamma(g,\nu)$ with each $\nu\in\mathcal M_G$, which is given by
\begin{equation}
\label{7.1}
\gamma(g,\nu)=\Re (Fg(i)-F(i)) ,
\end{equation}
with $F=F_{\nu}$ being the Herglotz function with measure $F$, as in \eqref{hergl}. This function is strictly
speaking not completely determined by $\nu$, but of course the unknown constant $a$ will drop out of the difference and thus is irrelevant here.

We can alternatively view $\gamma$ as a map on $\mathcal M (S_1)\times \ldots \times\mathcal M(S_N)\cong
\mathcal M(\mathcal F)$ (for fixed $g\in G$), using Theorem \ref{T6.1}, and then we write it as $\gamma(g,\nu_1,\ldots ,\nu_N)$.
We see from \eqref{7.1} that $\gamma(g,\nu)$ is a continuous $\R$-linear
functional of $\nu$. It is also a continuous function of $\nu_1,\ldots ,\nu_N$, by Theorem \ref{T6.1}, but of course not multilinear in the individual measures
in this version. In the sequel, we will again switch without further comment between these viewpoints, interpreting $\gamma$ and also $\Gamma$
from \eqref{Gamma} below as a map on $\mathcal M_G$, $\mathcal M(\mathcal F)$, or $\mathcal M (S_1)\times\ldots\times\mathcal M (S_N)$.

The homomorphism $\gamma(g)$ for a general $g\in G$ (and fixed $\nu$, for now) is determined by its values on a set of generators $g_1,\ldots , g_{N-1}$.
Fixing such a set of generators, we can then define a single continuous linear map
\begin{equation}
\label{Gamma}
\Gamma: \mathcal M _G\to\R^{N-1} , \quad \Gamma (\nu)
= \left( \gamma(g_1,\nu),\ldots ,\gamma(g_{N-1},\nu)\right) .
\end{equation}
We are interested in the question of when $\Gamma(\nu)=0$ for an automorphic measure $\nu$.
\begin{Lemma}
\label{L7.1}
Suppose that $\nu\in\mathcal M_G$, $\Gamma(\nu)=0$. If $\nu(S_n)=0$ for some $n=1,2,\ldots , N$, then $\nu=0$.
\end{Lemma}
Put differently, if the representing measure of an automorphic Herglotz function gives zero weight to an $I_n$, then $F\equiv a\in\R_{\infty}$.
\begin{proof}
Let's again suppose that $n=1$, for convenience.
Observe that the intervals adjacent to $I_1$ in Figure 1 are images $g\cdot I_1$ of the same interval. So if
$\nu(I_1)=0$, then in fact $\nu(J)=0$ for a larger open interval $J$ which contains $I_1$ and the neighboring intervals, which are also
mapped to $(a_1,b_1)$ by $\varphi$. The associated automorphic Herglotz function $F(\lambda)=\int_{\R_{\infty}}
\frac{1+t\lambda}{t-\lambda}\, d\nu(t)$ thus has a holomorphic continuation to $\C^+\cup J\cup\C^-$, and $\Im F(x)=0$ for $x\in J$.

We then deduce that the functions $m_{\pm}$ have similar behavior
near the gap $(a_1,b_1)$. Of course, we can not automatically conclude that they are holomorphic at the branch points $a_1,b_1$ of $\varphi$;
rather, these functions will have square root type behavior there. However, it is true and easy to see that $m_{\pm}$ have continuous extensions to
$\C^+\cup (a_1-\delta,b_1+\delta)$ and $m_{\pm}(x)\in\R$ for $x\in (a_1,b_1)$. This is impossible, basically because it contradicts \eqref{hfcn}.

Of course, \eqref{hfcn} as stated does not apply here because we did not assume that $H\in\mathcal D_0(U)$.
However, the reasoning that led to \eqref{hfcn} remains valid and gives the same behavior of $h$ locally. The point is that no matter what the value of $\mu_n$ is,
we cannot avoid $h$ being unbounded near some point of $[a_n,b_n]$. This would be obvious from \eqref{hfcn} if we had this formula available,
and it is still true in our much more general situation. Let me provide a somewhat more detailed sketch of this step.

First of all, observe that the Krein function $\xi$ of $h(z)=m_+(z)+m_-(z)$
does satisfy $\xi=1/2$ on $U$ if $h$ is not a constant from $\R_{\infty}$.
Moreover, the properties of the functions $m_{\pm}$ that were just derived imply that their measures and thus also the one
associated with $h$ give zero weight to $[a_1,b_1]$. In this situation, the Herglotz representation of $h$ shows that
\[
h'(x) = \int_{\R_{\infty}\setminus [a_1,b_1]} \frac{t^2+1}{(t-x)^2}\, d\mu(t) > 0
\]
on $x\in (a_1,b_1)$. This in turn implies that $\xi(x)=\chi_{(a_1,c)}(x)$ there, for some $a_1\le c\le b_1$. No matter what the value of $c$ is,
we have $\xi(x)=1$ near $a_1$ or $\xi(x)=0$ near $b_1$ (or both) for $x\in (a_1,b_1)$ and of course $\xi=1/2$ on the other side
of this point $t=a_1$ or $t=b_1$. In either case, this behavior of $\xi$ gives us a square root type singularity $h(z)\simeq (t-z)^{-1/2}$ there.
This contradicts what we learned about $m_{\pm}$ above. We have to admit that we are actually in the degenerate case $h\equiv a\in\R_{\infty}$,
and thus $\nu(\R_{\infty})=0$, as claimed.
\end{proof}
\begin{Corollary}
\label{C7.1}
Suppose that $\nu\in\mathcal M_G$, $\nu\not=0$, $\Gamma(\nu_1,\ldots, \nu_N)=0$. Then, if also $\Gamma(c_1\nu_1,\ldots ,c_N\nu_N)=0$, $c_n\ge 0$,
then $c_1=c_2=\ldots =c_N$.
\end{Corollary}
\begin{proof}
Denote the second measure, with restrictions $c_n\nu_n$, by $\mu\in\mathcal M_G$.
Let's say $c_1=\max c_n$. Consider $\rho=c_1\nu-\mu$. This measure is still positive since its restriction to $S_n$ is $\rho_n=(c_1-c_n)\nu_n$, and obviously it is automorphic
and $\Gamma(\rho)=0$. However, $\rho_1=0$, so Lemma \ref{L7.1} shows that $\rho=0$. Since also $\nu_n\not= 0$ for all $n$, by the same result, this implies
that $c_n=c_1$.
\end{proof}
We now turn to existence of solutions of $\Gamma (\nu)=0$.
\begin{Theorem}
\label{T7.1}
For any $\nu_n\in\mathcal M (S_n)$, $\nu_n(S_n)>0$, $n=1,2,\ldots , N$, there are unique constants $c_n>0$, $\sum c_n=1$,
such that
\[
\Gamma(c_1\nu_1,\ldots ,c_N\nu_N)=0 .
\]
\end{Theorem}
\begin{proof}
Uniqueness is guaranteed by Corollary \ref{C7.1}, so we only need to prove the existence of such $c_n$.
Of course, we need not pay any attention to the condition $\sum c_n=1$, which we can always satisfy by multiplying the $c_n$ by a constant.

We start with the case when $\nu_n=\delta_{x_n}$, $x_n\in S_n$. We mostly did this already, in the proof of Lemma \ref{L5.2}, though we focused on a single $x_n$
there. Let's briefly review the argument one more time: we take the (unique) $H\in\mathcal D_0(U)$ with parameters $\widehat{\mu}_n=(t_n,\sigma_n)$, $t_n=\varphi(x_n)$,
with $\sigma_n=\pm 1$ determined by whether $\varphi(x_n+iy)\in\Omega$ approaches $t_n\in (a_n,b_n)$ from above or below when $y\to 0+$. Then the measure
$\nu$ of the $F$ function of $H$ will satisfy $\Gamma(\nu)=0$ because $F$ is automorphic, and it
will have the desired point masses at the $x_n$. This was fairly obvious when $t_n\in (a_n,b_n)$, and, as we argued above, it is still true,
though less obvious, when $t_n=a_n$ or $=b_n$. Finally, the parametrization of $\mathcal D_0(U)$ that was reviewed following the
statement of Lemma \ref{L5.2} also makes it clear that $\nu(I_n\setminus\{ x_n\})=0$. This follows because if, say, $t_n\in (a_n,b_n)$ and $\sigma_n=1$,
then $m_-$ can be holomorphically continued through $(a_n,b_n)$ while $m_+$ is meromorphic there, with a single pole at $t_n$, and both functions
are real and continuous on $[a_n,b_n]\setminus \{ t_n\}$. The other cases lead to similar scenarios.

Next, suppose that, more generally, $\nu_1=\sum_{j=1}^J w_j \delta_{x_j}$, but still $\nu_n=\delta_{y_n}$ for $n\ge 2$, with $x_j\in S_1$, $y_n\in S_n$, $w_j>0$.
If we replace $\nu_1$ by just one of its point masses $w_j\delta_{x_j}$, then we are back in the first case, so we know there are
$d_2(j),\ldots , d_N(j)>0$ such that
\[
\Gamma(w_j\delta_{x_j},d_2(j)\delta_{y_2}, \ldots , d_N(j)\delta_{y_N}) = 0 .
\]
Thus $c_1=1$, $c_n=\sum_{j=1}^J d_n(j)$ works for the original measures.

Now we can repeat the same procedure in the second component, so we consider measures of the form
\begin{equation}
\label{7.31}
( \nu_1, \sum_{k=1}^K w_k \delta_{x_k} , \delta_{y_3}, \ldots , \delta_{y_N} ) ,
\end{equation}
where $\nu_1=\sum_{j=1}^J v_j \delta_{t_j}$ is of the type just discussed. By the previous step, we can now handle the measures where
we replace the second component by one of its summands $w_k\delta_{x_k}$ and then also the full measure from \eqref{7.31}, in the same way as above.
Continuing in this way, we establish the claim of Theorem \ref{T7.1} in the case
when all $\nu_n$ are finitely supported measures.

In general, if arbitrary measures $\nu_n$ are given, approximate them by such finitely supported measures $\nu^{(j)}_n\to\nu_n$. By what we just showed,
there are $c_n(j)> 0$, $\sum_{n=1}^N c_n(j)=1$, such that
\[
\Gamma \left( c_1(j)\nu^{(j)}_1,\ldots ,c_N(j)\nu^{(j)}_N\right) =0 .
\]
By passing to a subsequence,
we may assume that $c_n=\lim_{j\to\infty} c_n(j)$ exists. We have $c_n\ge 0$, $\sum c_n=1$. We observed at the beginning of this section
that $\Gamma(\nu)$ is continuous. Thus
\[
\Gamma(c_1\nu_1,\ldots ,c_N\nu_N)=0 ,
\]
as desired. Since $\sum c_n=1$, and, by assumption, $\nu_n(S_n)>0$, Lemma \ref{L7.1} guarantees that $c_n>0$ for all $n$.
\end{proof}
\section{Proof of Theorem \ref{T1.2}}
We'll prove the following slightly more detailed version of Theorem \ref{T1.2}.
\begin{Theorem}
\label{T8.1}
(a) $\{ F\in\mathcal H_G: F(i)=i\}$ is a compact convex set.

(b) Let $F=F_{\nu}=F(\cdot; H)$ be an automorphic Herglotz function from the set of part (a) with associated measure $\nu\in\mathcal M_G$.
Let $H\in\mathcal D(U)$ be the associated canonical system, that is, $F$ is the $F$ function of $H$.

Then the following are equivalent: (i) $F$ is an extreme point;
(ii) $H\in\mathcal D_0(U)$; (iii) The restrictions $\nu_n=\chi_{S_n}\nu$ are of the form $\nu_n=w_n\delta_{x_n}$, with $x_n\in I_n$, $w_n>0$.
\end{Theorem}
\begin{proof}
Part (a) is obvious from the already mentioned fact that $\mathcal H_G$ is compact and was only stated here because the corresponding claims were also
made in Theorem \ref{T1.2}.

A Herglotz function $F$ with representation \eqref{hergl} satisfies $F(i)=a+i\nu(\R_{\infty})$. In particular, if $F$ is from the set of part (a),
then $a=0$, $\nu(\R_{\infty})=1$.

To prove part (b), observe that the linear structure on the set from part (a) is the same as the one on the associated measures $\nu\in\mathcal M_G$.
More explicitly, we have $F_{c\mu+(1-c)\nu}=cF_{\mu}+(1-c)F_{\nu}$. So we may as well try to find the extreme points of the space of measures
\begin{equation}
\label{8.1}
X\equiv\{\nu\in\mathcal M_G: \nu(\R_{\infty})=1, \Gamma (\nu)=0 \} .
\end{equation}

Since for given $x_n\in I_n$, the $w_n$ from part (iii) are uniquely determined if we want to obtain a measure from this space, it is already clear
that these measures are extreme points.

Conversely, if $\nu\in X$ is not of the type described in (iii), then for at least one $n$, let's say for $n=1$,
we can write $\nu_1=\mu_1+\rho_1$, with $\mu_1,\rho_1\not= 0$ and singular with respect to each other. Now Theorem \ref{T7.1}
lets us find constants $c_2,\ldots , c_N>0$ such that
\[
\Gamma (\mu_1, c_2\nu_2, c_3\nu_3,\ldots , c_N\nu_N)=0 .
\]
Let's denote the corresponding measure from $\mathcal M_G$ by $\mu$. Of course, we can do the same thing for $\rho_1$, and we obtain a second
measure $\rho\in\mathcal M_G$, satisfying $\Gamma (\rho)=0$,
and with restriction to $S_1$ equal to $\rho_1$, and the other restrictions, to $S_n$, are multiples of the $\nu_n$.
Now $\sigma=\mu+\rho$ satisfies $\Gamma (\sigma)=0$, $\sigma_1=\nu_1$, $\sigma_n=d_n\nu_n$ ($n\ge 2$). Hence $\sigma=\nu$ by Corollary \ref{C7.1}.
We have succeeded in writing
\[
\nu = \mu(\R_{\infty}) \frac{\mu}{\mu(\R_{\infty})} + \rho(\R_{\infty}) \frac{\rho}{\rho(\R_{\infty})}
\]
as a non-trivial convex combination. Notice here that $\mu(\R_{\infty}),\rho(\R_{\infty})>0$, and indeed
\[
\mu(\R_{\infty})+\rho(\R_{\infty}) = \nu(\R_{\infty})=1 ,
\]
as required. Obviously, the normalized versions of $\mu,\rho$ are in $X$. Furthermore, $\mu\not= \rho$ by construction.
We have shown that only the measures listed in (iii) are extreme points.

Finally, we already discussed earlier that exactly these measures give us the finite gap operators $H\in\mathcal D_0(U)$.
\end{proof}
\section{Representation of reflectionless operators by measures}
We already have a homeomorphism $\mathcal D(U)\cong \{ F\in\mathcal H_G: F(i)=i\}$ from Theorem \ref{T1.1}.
The correspondence between a Herglotz function and the data $(a,\nu)$ from its representation \eqref{hergl} is a homeomorphism also;
see, for example, \cite[Theorem 7.3]{Rembook}. So in our current situation, we have
a homeomorphism $\mathcal D(U)\cong X$, with $X$ still denoting the space of measures from \eqref{8.1}. 
In this version, the statement contains a condition involving the map $\Gamma$; we
can give it a more satisfying form with the help of Theorems \ref{T6.1}, \ref{T7.1}.
\begin{Theorem}
\label{T9.1}
The following map is a homeomorphism:
\begin{gather*}
\mathcal D(U) \to\mathcal M_1(S_1)\times\ldots\times\mathcal M_1(S_N) ,\\
H \mapsto \left( \frac{\nu_1}{\nu_1(S_1)} , \ldots, \frac{\nu_N}{\nu_N(S_N)}\right) .
\end{gather*}
It sends an $H$ to the normalized versions of the restrictions to $S_n$ of the automorphic measure associated with the $F$ function of $H$.
\end{Theorem}
\begin{proof}
The preceding discussion has already established the variant version of this statement where we don't normalize the measures.
Next, notice that Theorem \ref{T7.1} lets us recover the numbers $\nu_n(S_n)>0$ from the normalized measures $\mu_n=\nu_n/\nu_n(S_n)$
since $\Gamma(\nu_1, \ldots ,\nu_N)=0$ and $\nu(\R_{\infty})=1$.
Lemma \ref{L7.1} guarantees that indeed $\nu_n(S_n)>0$.

Clearly the maps $\nu\mapsto \nu_n(S_n)$ are continuous, by Theorem \ref{T6.1}, so the map from Theorem \ref{T9.1} has now been recognized
as a continuous bijection. We are mapping between compact metric spaces, so the continuity of the inverse map is automatic.
\end{proof}

We can elaborate some more on this theme. Since we are mapping from a compact space, we in fact have
\[
m_n\equiv\min_{H\in\mathcal D(U)} \nu_n(S_n)>0 .
\]
Moreover, this map is linear with respect to the natural linear structure on $F$ functions satisfying $F(i)=i$, so, as already mentioned in the introduction,
Theorem \ref{T1.2} now shows that $m_n= \min_{H\in\mathcal D_0(U)} \nu_n(S_n)$. Or, if we use the more explicit description from Theorem \ref{T8.1}(b)(iii),
then we can say that $m_n=\min w_n$ where now the minimum is taken over the torus $(x_1,\ldots, x_N)\in S_1\times\ldots\times S_N$ and, given such points,
the $w_n>0$ denote the unique weights for which the automorphic measure $\nu$ with restrictions $(w_1\delta_{x_1},\ldots ,w_N\delta_{x_N})$ satisfies
$\Gamma(\nu)=0$, $\nu(\R_{\infty})=1$.

Of course, similar remarks apply to the maximum of $\nu_n(S_n)$.

Finally, while most of our results dealt with Dirac operators exclusively, we can easily go back to the more general setting of canonical systems, if this is desired,
for example by using the same simple device that was also employed in \cite{ForRem,RemZ} in similar situations. Please see these references for more on this theme;
we limit ourselves to a few quick and mostly obvious remarks here. As a preparation, introduce the notation
\[
\mathcal Z = \{ H(x)\equiv P_{\alpha}: 0\le\alpha <\pi \}
\]
for the trivial canonical systems with constant (extended) real $m$ functions. So the $F$ function of $H(x)\equiv P_{\alpha}$ is $F(\lambda)\equiv -\tan\alpha\in\R_{\infty}$.
Then $\mathcal R(U)\setminus\mathcal Z$ can be restored from $\mathcal D(U)$ by letting the translation/dilation group $w\mapsto cw+a$, $c>0$, $a\in\R$,
act on $m$ functions, or, what is the same here, on $F$ functions. So we similarly map $F$ to $cF+a$. I mention in passing that there is also an easy explicit description
of what this action does to the coefficient functions $H(x)$ \cite[Theorem 3.20]{Rembook}.

Obviously, we recover all $F$ functions except the trivial
ones $F\equiv b\in\R_{\infty}$ from above if we act in this way on the $F$ functions satisfying $F(i)=i$. Since the acting group is homeomorphic to $\C^+$,
this gives a natural identification $\mathcal R(U)\setminus\mathcal Z\cong \C^+\times\mathcal D(U)$, and allows us to make use of the maps constructed above,
for example in Theorem \ref{T9.1}, for general canonical systems also. In the same way, we obtain a natural correspondence
$\mathcal R_0(U)\setminus\mathcal Z\cong\C^+\times\mathcal D_0(U)$.

\end{document}